\documentclass{amsart}

\theoremstyle{plain}
\newtheorem{theorem}{Theorem}
\newtheorem{corollary}[theorem]{Corollary}

\newtheorem{proposition}[theorem]{Proposition}
\theoremstyle{definition}
\newtheorem{example}[theorem]{Example}

\newtheorem{definition}[theorem]{Definition}

\newtheorem{remark}[theorem]{Remark}
\theoremstyle{remark}

\begin{document}
\title{Kaplansky/Nagata-type  theorems for the half factorial domains}

\author{Tiberiu Dumitrescu}

\address{Facultatea de Matematica si Informatica,University of Bucharest,14 A\-ca\-de\-mi\-ei Str., Bucharest, RO 010014,Romania}
\email{tiberiu\_dumitrescu2003@yahoo.com }
\maketitle


\begin{abstract}
We give   Kaplansky/Nagata-type theorems for the half factorial domains  inside the class of atomic domains.
\end{abstract}
\vspace{1cm}
Kaplansky's \cite[Theorem 5]{K}  states that an (integral) domain $A$ is a unique factorization domain  (UFD) iff every nonzero prime ideal contains a prime element. 
Subsequently, a theorem having this form  was called  in the literature a {\em Kaplansky-type theorem}. An example is \cite[Theorem 4.1]{Ka}  stating that a domain $A$ is a $\pi$-domain iff every nonzero prime ideal of $A$ contains an invertible prime ideal. Recall that a {\em $\pi$-domain} is a domain  whose proper principal ideals are products of prime ideals.
For examples  of Kaplansky-type theorems see \cite{AZ}, \cite{CK1}, \cite[page 284]{K}, \cite{Ki} and their reference list.
Analyzing the proof of \cite[Theorem 5]{K} (as done   in  \cite{AZ}) we get the following   abstract form. 

\begin{theorem} \label{6}
Let $A$ be a domain and  $\Gamma\subseteq A-\{0\}$.

$(i)$ If every nonzero prime ideal of $A$ cuts    $\Gamma$, then every   nonzero element of $A$ divides a product of elements of $\Gamma$.

$(ii)$ If $I$ is a nonzero ideal of $A$  such that 
$V(I)\subseteq \cup_{x\in \Gamma} V(x)$, then $I$ contains a product of elements of $\Gamma$. As usual $V(I)$  is the set of prime ideals containing $I$.
\end{theorem}
\begin{proof}
For reader's convenience we insert a proof of $(ii)$. If the conclusion  fails, then $I$ is disjoint of  the multiplicative set $S$ generated by $\Gamma$, so expanding $I$ to a prime ideal disjoint of $S$ we get a contradiction.
\end{proof}

Nagata's Theorem for UFDs \cite{N} states  that  a Noetherian domain $A$ is  a UFD provided there exists a    multiplicative set $S\subset A$ generated by prime elements such that the fraction ring $A_S$ is a UFD. 
Subsequently, a theorem of this form  was called    a {\em Nagata-type theorem}.   For examples of Nagata-type theorems  see
\cite[Section 3]{AAZ} and its reference list.

Let $A$ be a domain. A nonzero nonunit   $x\in A$ is called an {\em irreducible} element  (aka {\em atom}) if $x$ is not a product of two nonunits of $A$. According to  \cite{C}, $A$ is  said to be {\em  atomic }  if each nonzero nonunit $x$ of $A$ is a product of atoms (i.e. $x$ has an atomic factorization). If $A$ is atomic and for each nonzero nonunit $x\in A$  all  atomic factorizations of $x$  have the same  length (i.e. number of factors), then  $A$ is called a {\em half-factorial domain (HFD)}, cf. \cite{Z}. The result which triggered the study of HFDs was Carlitz's Theorem \cite{Ca} which states that a ring of algebraic integers is an HFD iff  
a product of  every two nonprincipal ideals of $A$ is   principal (i.e. 
the class group of $A$ has order $\leq 2$). 
For literature on HFDs see \cite{Co}, \cite{CC} and the references there.

The aim of this paper is to give  Kaplansky/Nagata-type theorems for HFDs inside the class of atomic domains (thus giving a partial answer to a question posed by Jim Coykendall in a recent conference talk).
After introducing our "good/bad atoms" concept   (Definition \ref{8})
we give: a Kaplansky-type theorem for HFDs  (Theorem \ref{10}), some  facts about good/bad atoms (Remark \ref{5}, Proposition \ref{4}, Example \ref{9}), a Nagata-type theorem for HFD's (Theorem \ref{7}  with  its polynomial consequence Corollary \ref{11}) and 
a possible extension of these ideas to $r$-congruence half-factorial domains (Theorems \ref{12}, \ref{13} and  Corollary \ref{14})). We use   books \cite{G} and \cite{K} as general reference for basic results and terminology.

We introduce the key concept of this paper.

\begin{definition} \label{8}
Let $A$ be an atomic domain. 
By an {\em atomic equality} we mean 
an equality  
$$(\sharp) \hspace{3cm} a_1\cdots a_n = b_1\cdots b_m \hspace{5cm} $$  whose factors are atoms in $A$. 

$\bullet$ Call   atomic equality $(\sharp)$ {\em balanced} (resp. {\em unbalanced}) if $n=m$ (resp. $n\neq m$).

$\bullet$ Call   atomic equality $(\sharp)$ {\em irredundant} if  no
proper subproduct of the $a$'s  is associated  to a proper subproduct of the $b$'s, that is, $\prod_{i\in I} a_i$ is not  associated to $\prod_{j\in J} b_j$ whenever 
$I\subset \{1,...,n\}$, $J\subset \{1,...,m\}$ are nonempty sets.
Thus  a {\em redundant} (i.e. non-irredundant) atomic equality  splits essentially (i.e. up to unit multiplication) into two smaller atomic equalities.

$\bullet$ Designate  the atoms appearing in irredundant unbalanced atomic equalities  as {\em bad atoms} and call
the others  {\em good atoms}.
\end{definition}


\begin{remark} \label{5}
Let $A$ be an atomic domain.

$(i)$ A prime atom  is clearly good. There exist nonprime good atoms as seen in  Example \ref{9}.

$(ii)$ $A$ is an HFD iff its atoms    are   good, because  an unbalanced atomic equality may be shrunk to an irredundant unbalanced  one.

$(iii)$ As $3^4=(5+ 2\sqrt{-14})(5- 2\sqrt{-14})$ is an irredundant unbalanced atomic equality in $\mathbb{Z}[\sqrt{-14}]$, the factors are bad atoms.  In fact all nonprime atoms of $\mathbb{Z}[\sqrt{-14}]$ are bad, cf. Proposition \ref{4}.

$(iv)$ A good atom  $a\in A$ may be  bad in $A[X]$. 
For instance, $\mathbb{Z}[\sqrt{-3}]$  is an HFD (cf. \cite{Z}) and but  $2$ is a bad atom of $\mathbb{Z}[\sqrt{-3}][X]$ due to atomic equality
$$2^2(X^2+X+1)=(2X+1+\sqrt{-3})(2X+1-\sqrt{-3}).$$

\end{remark}

\vspace{2mm}
 We give a Kaplansky-type theorem for the half factorial domains. Let $A$ be an atomic domain.
 Say that  a nonzero nonunit   $x\in A$   is an {\em HF-element} if all atomic factorizations of $x$ have the same length. So $A$ is an HFD iff its nonzero nonunits are HF-elements.

\begin{theorem} \label{10}
An atomic domain $A$  is an HFD iff every nonzero prime ideal of $A$ contains a good atom.
\end{theorem}
\begin{proof}
$(\Rightarrow)$ Every nonzero prime ideal contains some atom which is good since $A$ is an HFD.
$(\Leftarrow)$\  
Denote by  $S$   the multiplicative set generated by all good atoms.
Note that every  $s\in S$ is an HF-element.  
Indeed, a hypothetic  unbalanced atomic equality  $a_1\cdots a_n = b_1\cdots b_m$   with $a_1,...,a_n$ good atoms can be shrunk to an irredundant one of the same kind,  a contradiction.
If  $x$ is a nonzero nonunit of $A$, then part $(i)$ of Theorem \ref{6} shows that $x$ divides 
some  $y\in S$, so $x$ is an HF-element because $y$ is so (as noted  above). Thus $A$ is an HFD.
\end{proof}

Combining  the  proof above  with part $(ii)$ of Theorem \ref{6} we get:

\begin{corollary}
If  $x$ a nonzero nonunit  of   an atomic domain  $A$ and
  every prime ideal containing $x$ contains a good atom, 
then $x$   
is an HF-element.
\end{corollary}

\begin{corollary}
In a  quasilocal one-dimensional atomic domain, all atoms are good or all atoms are bad.
\end{corollary}

The next result shows that   a ring of algebraic integers  having a good nonprime  atom  is an HFD.

\begin{proposition}\label{4}
For a  Krull domain  $A$, consider the  following assertions:

$(i)$ The divisor class group of $A$ has order $\leq 2$.

$(ii)$ $A$ is   an HFD.

$(iii)$ $A$ is a UFD or it has at least a  good nonprime  atom.
\\
Then $(i)$ $\Rightarrow$ $(ii)$  $\Rightarrow$ $(iii)$.
If  each divisor class of $A$ contains a height-one prime ideal 
(e.g. $A$ is a ring of algebraic integers, cf. \cite{Na}), then  all three assertions are equivalent.
\end{proposition}
\begin{proof}
$(i)$ $\Rightarrow$ $(ii)$ is Zaks' \cite[Theorem 1.4]{Z} extending Carlitz's Theorem \cite{Ca} while
$(ii)$ $\Rightarrow$ $(iii)$ is clear.
To complete the proof, we show that
$(iii)$ implies $(i)$ provided each divisor class of $A$ contains a height-one prime ideal. Suppose that the divisor class group of $A$ has order $\geq 3$. We show that every   nonprime atom $x$ of $A$ is bad.
Write $xA$  as a $v$-product 
$xA=(P_1\cdots P_n)_v$, $n\geq 2$, of height-one prime ideals.  
For a divisorial ideal $I$ denote its divisor class by $[I]$.
Suppose that $n\geq 3$. Pick a height-one prime ideal $Q_i$ inside $-[P_i]$ for $i=1,...,n$. We obtain the principal ideals generated by atoms: $(Q_1\cdots Q_n)_v=zA$ and $(P_iQ_i)_v=y_iA$, $i=1,...,n$. So $x$ is a bad atom  because 
$$ xzA = (P_1\cdots P_nQ_1\cdots Q_n)_v = y_1\cdots y_nA.
$$
Finally suppose that $xA=(P_1P_2)_v$. 
Let $c=[P_1]$ and pick $d$ a nonzero divisor class distinct from $c$ (a possible choice since     $A$ has at least three divisor classes).
Pick   prime ideals $Q_1$, $Q_2$, $Q_3$, $Q_4$  in  divisor classes
$-d$, $d-c$, $d$, $c-d$ respectively. We obtain the principal ideals generated by atoms: $(P_1Q_1Q_2)_v=a_1A$, $(P_2Q_3Q_4)_v=a_2A$, $(Q_1Q_3)_v=a_3A$, 
$(Q_2Q_4)_v=a_4A$.
Then $x$ is a bad atom  because 
$$ xa_3a_4A = (P_1P_2Q_1Q_2Q_3Q_4)_v  = a_1a_2A.
$$
\end{proof}

We exhibit a Dedekind non-HFD having  good nonprime atoms.

\begin{example} \label{9}
By Grams' theorem \cite[Corollary 1.5]{Gr}, there exists a Dedekind domain $A$ with class group $\mathbb{Z}/6\mathbb{Z}$  such that the   ideal classes which contain  maximal ideals are precisely $\widehat{2}$, $\widehat{3}$ and $\widehat{4}$. 
For the nonprime atoms $x$   of $A$ we have the following types (where each $P_i$ is a  height-one prime ideal):

$(1)$  $xA=P_1P_2$ with $P_1,P_2\in \widehat{3}$,

$(2)$  $xA=P_1P_2P_3$ with $P_1,P_2,P_3\in \widehat{2}$,

$(3)$  $xA=P_1P_2P_3$ with $P_1,P_2,P_3\in \widehat{4}$,

$(4)$  $xA=P_1P_2$ with $P_1\in \widehat{2}$ and $P_2\in \widehat{4}$.
\\
Let $ab=a'b'$ where $a$, $a'$ (resp. $b$, $b'$) are products of type one  (resp. type $\geq 2$) atoms. As $A$ has unique prime ideal factorization, we get $aA=a'A$ and $bA=b'A$. It   easily follows  that   type one atoms  are good while the others are bad (e.g. a type two atom times a type three one is a product a three type four atoms).
\end{example}

 We give a Nagata-type theorem for HFDs (compare to \cite[Theorem 3.3]{AAZ}).

\begin{theorem} \label{7}
Let $A$ be an atomic domain and $S\subseteq A$ a saturated multiplicative set
whose elements are products of good atoms.
If $A_S$ is an HFD, then so is $A$.
\end{theorem}
\begin{proof}
 Let $B$ be the set of  good atoms in $S$ and $C$ the set of  atoms in $A-S$ which are also atoms  in $A_S$.
We prove first the following claim.

{\em $(*)$\ \   Every   atom    $g\in A-(B\cup C)$ may be inserted into a balanced atomic equality
$$ a_1a_2\cdots a_n g = f_1f_2\cdots f_{n+1}
$$
where the $a$'s are  in $B$ and the $f$'s are in $B\cup C$.} 

Perform the following  operations: $(1)$ consider an atomic factorization of $g$ in $A_S$,  $(2)$  multiply it by  atoms in $B$ in order to push all  factors in $A$ and  $(3)$ write atomic factorizations for all resulting factors in $A-S$.
Note that each $b\in A$  which is an atom in $A_S$  can be written as a product of atoms from $B$ times an atom from $C$. 
 We thus achieve an atomic equality 
$a_1a_2\cdots a_n g = f_1f_2\cdots f_{m}
$
where the $a$'s are in $B$ and the $f$'s are in $B\cup C$. Since the $a$'s are good atoms and $g\notin B\cup C$, we may arrange that $m=n+1$, thus proving the claim. Now we prove that $A$ is an HFD.  Suppose, by contrary,  there exists an unbalanced atomic equality  
$$(**) \hspace{3cm} g_1g_2\cdots g_j=h_1h_2\cdots h_k. \hspace{4cm} $$
 Using Claim $(*)$, we may assume that all factors in $(**)$ are from $B\cup C$. Since $A_S$ is an HFD, at least one of the factors in $(**)$ is a (good) atom from $B$. Then   $(**)$ is  a redundant atomic equality,  
so it splits essentially  into two atomic equalities with fewer factors, at least one of them being unbalanced. An inductive argument provides the desired contradiction.  
\end{proof}

In general, a polynomial extension of a HFD is not a HFD (see
 part $(iv)$ of Remark \ref{5}).

\begin{corollary} \label{11}
If $D$ is an HFD  whose atoms are good in $D[X]$, then $D[X]$ is an HFD.
\end{corollary}
\begin{proof}
It is well-known and anyway easy to see that $D[X]$ is atomic.
Apply Theorem \ref{7} for $A=D[X]$ and $S=D-\{0\}$.
\end{proof}

\begin{remark} 
The nontrivial implication of Theorem \ref{10} is a consequence of  Theorem  \ref{7}. Indeed, let $A$ be an atomic and $S$   the   multiplicative set genereted by all good atoms. If 
all nonzero prime ideals of $A$ cut $S$, then $A_S$ is the quotient field of $A$,  so $A$ is an HFD by Theorem  \ref{7}.
\end{remark}

In \cite[Theorem 2.10]{CS}, Coykendall and   Smith proved that any {\em other half factorial domain (OHFD)} is a UFD, where an OHFD is an atomic domain in which any two distinct factorizations of the same nonzero nonunit element have different lengths (see also \cite{{CCGS}}).
As another application of Theorem \ref{7}, we give an alternative (but more complicated) proof of that result.

\begin{theorem} {\em (\cite[Theorem 2.10]{CS})}
Any OHFD is an UFD.
\end{theorem}
\begin{proof}
Suppose there exists   an OHFD $A$ which is not an UFD. Note in an OHFD 
all good atoms  are  prime and let $S$ be the multiplicative set 
they generate.
By \cite[Corollary 2.2]{AAZ}, $A_S$ is atomic  and it is easy to see that $A_S$ is still an OHFD.
By  Theorem \ref{7}, we may replace $A$ by $A_S$ and thus assume  that   $A$ has only bad atoms.  By \cite[Proposition 2.3]{CS}, $A$ is a Cohen-Kaplansky domain, i.e. it has finitely many nonassociate atoms. 
By  \cite[Theorem 7]{CK} and  \cite[Theorem 6.1]{AM}
we may assume that $A$ is quasilocal. Let $a$, $b$ be two distinct atoms of $A$. As the integral closure $\bar{A}$ of $A$ is a DVR and $A\subseteq \bar{A}$ is a root extension (cf. \cite[Theorem 4.3]{AM}), a short computation shows that $a^n$ is an associate of $b^m$ for some $n\neq m$. By \cite[Proposition 2.3]{CS}, $a$ and $b$ are the only atoms of $A$, a contradiction cf. \cite[Introduction]{CK}.
\end{proof}

Let $A$ be an atomic domain and $r>1$ an integer. Following \cite{CS1} 
we say that $A$ is  an {\em $r$-congruence half-factorial domain ($r$-CHFD)} if for each nonzero nonunit $x$ of $A$, the lengths of all  atomic factorizations of $x$ are congruent modulo $r$. 
Update Definition \ref{8} as follows.

\begin{definition}  
Let $A$ be an atomic domain, $r>1$ an integer and  
$$(\sharp) \hspace{3cm} a_1\cdots a_n = b_1\cdots b_m \hspace{5cm} $$    an   atomic equality. 

$\bullet$ Say that   atomic equality $(\sharp)$ is {\em $r$-balanced} if $n\equiv_r m$ (otherwise call it {\em $r$-unbalanced}).

$\bullet$ Designate  the atoms appearing in irredundant $r$-unbalanced atomic equalities  as {\em $r$-bad atoms} and call
the others  {\em $r$-good atoms}.
\end{definition}

We invite the reader to adapt the material above in order to prove the following three results.

\begin{theorem} \label{12}
Let $A$ be an atomic domain and $r>1$ an integer. Then $A$ is an $r$-CHFD iff every nonzero prime ideal of $A$ contains an $r$-good atom.
\end{theorem}

\begin{theorem} \label{13}
Let $A$ be an atomic domain, $r>1$ an integer and $S\subseteq A$ a saturated multiplicative set
whose elements are products of $r$-good atoms.
If $A_S$ is an $r$-CHFD, then so is $A$.
\end{theorem}

\begin{corollary} \label{14}
If $D$ is an $r$-CHFD  whose atoms are $r$-good in $D[X]$, then $D[X]$ is an $r$-CHFD, provided $D[X]$ is atomic.
\end{corollary}


\begin{thebibliography}{11111}

\bibitem{AAZ} { D.D. Anderson, D.F. Anderson and M. Zafrullah,}
{Factorization in integral domains II,}
J. Algebra 152 (1992), 78-93.


\bibitem{AM} D.D. Anderson and J.L. Mott,  
Cohen-Kaplansky domains: Integral domains with a finite 
number of irreducible elements,
J. Algebra 148 (1992), 17-41 . 

\bibitem{AZ} { D.D. Anderson and M. Zafrullah,}
{On a theorem of Kaplansky,}
Boll. Un. Mat. Ital.   (7) 8-A (1994), 397-402.

\bibitem{Ca} L. Carlitz, A characterization of algebraic number fields with class number two, Proc. Amer. Math. Soc. 11 (1960), 391-392.

\bibitem{CK1} G.W. Chang and  H. Kim, 
Kaplansky-type theorems II, 
Kyungpook Math. J. 51 (2011),  339-344. 

\bibitem{CC} S.T. Chapman and J. Coykendall, 
Half-factorial domains, a survey. (in Non-Noetherian commutative ring theory (Chapman et al. ed.), Dordrecht: Kluwer Academic Publishers. Math. Appl., Dordr. 520 (2000), 97-115. 

\bibitem{CCGS} S.T. Chapman, J. Coykendall, F. Gotti, and W.W. Smith, 
Length-factoriality in commutative monoids and integral domains, 
J. Algebra 578 (2021), 186-212.

\bibitem{CS1} S.T. Chapman  and  W.W. Smith, 
On a characterization of algebraic number fields with class number less than three,
J. Algebra 135 (1990), 381-387. 

\bibitem{CK} I.S. Cohen and I. Kaplansky, 
Rings with a finite number of primes I, 
Trans. Amer. Math. Soc. 
60 (1946), 468-477.

\bibitem{C} P.M. Cohn,  Bezout rings and their subrings,
Proc. Cambridge Philos. Soc. 64 (1968), 251-264.

\bibitem{Co} J. Coykendall,  
Extensions of half-factorial domains, a survey (S.T. Chapman ed.), Arithmetical properties of commutative rings and monoids. Boca Raton, FL: Chapman and Hall/CRC, Lecture Notes in Pure and Applied Mathematics 241 (2005, 46-70.



\bibitem{CS} J. Coykendall and  W.W. Smith, On unique factorization domains, J. Algebra 332 (2011), 62-70. 

\bibitem{G} R. Gilmer, {\em Multiplicative Ideal Theory}, Marcel Dekker, New York, 1972.

\bibitem{Gr} A. Grams, The distribution of prime ideals of a Dedekind domain, Bull. Aust. Math. Soc. 11 (1974), 429-441.

\bibitem{Ka} B.G. Kang,  
On the converse of a well-known fact about Krull domains, 
J. Algebra 124 (1989), 284-299. 

\bibitem{K} I. Kaplansky,
{\em Commutative Rings}, rev. ed. The University of Chicago Press,
Chicago and London, 1974.

\bibitem{Ki} H. Kim, 
Kaplansky-type theorems,  
Kyungpook Math. J. 40 (2000),  9-16.

\bibitem{N} M. Nagata, A remark  on the unique factorization theorem, J. Math. Soc. Japan 9 (1957), 143-145.

\bibitem{Na} W. Narkiewicz, {\em Elementary and Analytic Theory of Algebraic Numbers}, 3rd edition, Springer 2004.

\bibitem{Z} A. Zaks, Half-factorial domains, Israel J. Math. 37 (1980), 281-302.

\end{thebibliography}
\end{document}